\documentclass[12pt]{amsart}

\usepackage{amsmath}
\usepackage{amssymb}  % defines \nmid, for example
\usepackage{latexsym} % for \Box
\usepackage{comment}
\usepackage{url}
\usepackage{fullpage,url,amssymb,amsmath,amsthm,amsfonts,mathrsfs}
\usepackage[usenames,dvipsnames]{color}
\usepackage[pagebackref = true, colorlinks = true, linkcolor = blue, citecolor =
 Green]{hyperref}
\usepackage[alphabetic,lite,nobysame]{amsrefs}
\usepackage{enumitem}
\usepackage{amscd}   % for commutative diagrams
\usepackage[all, cmtip]{xy} % for complicated commutative diagrams
\usepackage{xfrac}
\usepackage[T1]{fontenc}
\usepackage[title]{appendix}

\DeclareFontEncoding{OT2}{}{} % to enable usage of cyrillic fonts

\usepackage[all]{xy}
\usepackage{fullpage}

\usepackage{color} 
\newcommand{\defi}[1]{\textsf{#1}} % for defined terms

\def\act#1#2%
  {\mathop{}%
   \mathopen{\vphantom{#2}}^{#1}%
   \kern-3\scriptspace%
   #2}

\newcommand{\Z}{{\mathbb Z}}
\newcommand{\N}{{\mathbb N}}
\newcommand{\Q}{{\mathbb Q}}

\newcommand{\F}{{\mathbb F}}
\newcommand{\A}{{\mathbb A}}

\newcommand{\Kbar}{{\overline{K}}}

\DeclareMathOperator{\Sel}{Sel}

\DeclareMathOperator{\End}{End}

\DeclareMathOperator{\Gal}{Gal}

\DeclareMathOperator{\Res}{Res}
\DeclareMathOperator{\Br}{Br}

\DeclareMathOperator{\Jac}{Jac}
\DeclareMathOperator{\HH}{H}

\DeclareMathOperator{\Spec}{Spec}

\DeclareMathOperator{\red}{red}

\def\sep{{\rm sep}}

\newtheorem{Theorem}{Theorem}[section]
\newtheorem{Lemma}[Theorem]{Lemma}
\newtheorem{Proposition}[Theorem]{Proposition}

\newtheorem{Definition}[Theorem]{Definition}

\newtheorem{Remark}[Theorem]{Remark}

\theoremstyle{definition}

\numberwithin{equation}{section}

\begin{document}
\title{The Brauer-Manin obstruction for nonisotrivial curves over global function fields}
\author{Brendan Creutz}
\address{School of Mathematics and Statistics, University of Canterbury, Private
 Bag 4800, Christchurch 8140, New Zealand}
\email{brendan.creutz@canterbury.ac.nz}
\urladdr{http://www.math.canterbury.ac.nz/\~{}b.creutz}

\author{Jos\'e Felipe Voloch}
\address{School of Mathematics and Statistics, University of Canterbury, Private
 Bag 4800, Christchurch 8140, New Zealand}
\email{felipe.voloch@canterbury.ac.nz}
\urladdr{http://www.math.canterbury.ac.nz/\~{}f.voloch}

\author{an appendix by Damian R\"ossler}
\address{University of Oxford,
Andrew Wiles Building,
Radcliffe Observatory Quarter,
Woodstock Road,
Oxford OX2 6GG,
United Kingdom}
\email{damian.rossler@maths.ox.ac.uk}
\urladdr{http://people.maths.ox.ac.uk/rossler/}

%\makeatletter
%\let\@wraptoccontribs\wraptoccontribs
%\makeatother
%\contrib[with an appendix by]{Damian R\"ossler}

%%
\begin{abstract}
We prove that the set of rational points on a nonisotrivial curves of genus at least $2$ over a global function field is equal to the set of adelic points cut out by the Brauer-Manin obstruction.
\end{abstract}

\maketitle

\section{Introduction}

Let $X/K$ be a smooth projective and geometrically irreducible curve of genus at least $2$ over a global field $K$ of characteristic $p > 0$. 
We prove that if $X$ is not isotrivial, then the Brauer-Manin obstruction cuts out exactly the set of rational points on $X$.

\begin{Theorem}\label{thm:BM}
	Let $X/K$ be a smooth projective curve of genus at least $2$ over a global function field $K$. If $X$ is not isotrivial, then $X(\A_K)^{\Br} = X(K)\,.$
\end{Theorem} 

We refer the reader to \cite{PoonenVoloch} for the definition of the Brauer-Manin obstruction and the relevant background in this context. Theorem~\ref{thm:BM} is proved in that paper for $X$ contained in an abelian variety $A$ such that  $A(K^{\textup{sep}})[p^\infty]$ is finite and no geometric isogeny factor of $A$ is isotrivial. That result holds more generally for any `coset free' subvariety of an abelian variety over $K$. We remove the hypotheses on an abelian variety containing $X$, but our proof does not immediately extend to higher dimensional subvarieties of abelian varieties.

As in \cite{PoonenVoloch} our results are a consequence of related results concerning adelic intersections whose connection to the Brauer-Manin obstruction was first observed in \cite{Scharaschkin} for curves over number fields.

\begin{Theorem}\label{thm:MW}
	Suppose $X \subset A$ is a proper smooth curve of genus at least $2$ contained in an abelian variety $A$ over a global field $K$ of characteristic $p > 0$. Then $X(\A_k) \cap \overline{A(K)} = X(K)$, where $\overline{A(K)}$ denotes the topological closure of $A(K)$ in $A(\A_K)$.
\end{Theorem}

We follow the strategy of the proof in \cite{PoonenVoloch}, but there are two new ingredients allowing us to remove all hypotheses on an abelian variety containing $X$. The first, appearing as Proposition~\ref{prop:main} below, is based on ideas in the proof of the Mordell-Lang Conjecture appearing in \cite{AV,V}. This replaces the input in \cite{PoonenVoloch} from \cite{MR1333294} which relies heavily on model theory and requires assumptions on the Jacobian of $X$. The second new ingredient is an isogeny constructed by R\"ossler in the appendix to this paper. We use this instead of multiplication by $p$ in some of the arguments appearing in \cite{PoonenVoloch} to prove Proposition~\ref{prop:ZsubA}. This removes the need for the hypothesis on $A(K^{\textup{sep}})[p^\infty]$ in~\cite[Proposition 5.3]{PoonenVoloch} and elsewhere. 

The theorems above are expected to hold (in a slightly modified form) for any closed subvariety of an abelian variety over a global field. This was originally posed as a question in the case of curves over number fields by Scharaschkin \cite{Scharaschkin} and, independently, by Skorobogatov \cite{Skorobogatov}. It was later stated as a conjecture for curves over number fields in \cite{Poonen} and \cite{Stoll}. The number field case has seen little progress and remains wide open. Building on \cite{PoonenVoloch}, this paper settles the function field analogue of these conjectures for nonisotrivial curves of genus $\ge 2$. Some partial results toward the conjecture in the isotrivial case are given in \cite{CV,CV2}, but this case too remains open.

\section{Zariski dense adelic points surviving $p^{\infty}$-descent}

In this section we assume $X \subset A$ is a proper smooth curve of genus $\ge 2$ contained in an abelian variety $A$ over $K$.

\begin{Definition}
	Let $N\ge 1$ be an integer. An $N$-covering of a subvariety $X \subset A$ of an abelian variety $A$ over $K$ is an fppf-torsor $Y\to X$ under the $N$-torsion subgroup scheme $A[N]$ such that the base change of $Y\to X$ to $K^\textup{sep}$ is isomorphic to the pull back of multiplication by $N$ on $A$. An adelic point on $X$ is said to survive $N$-descent if it lifts to an adelic point on some $N$-covering of $X$.
\end{Definition}

\begin{Definition}
An adelic point $(P_v)_v \in X(\A_K)$ is called Zariski dense if
for any proper closed subvariety $Y \subsetneq X$, there exists $v$ such that
$P_v \notin Y$.
\end{Definition}

\begin{Proposition}\label{prop:main}
Suppose $X \subset A$ is a proper smooth curve of genus at least $2$ contained in an abelian variety $A$ over a global field $K$ of characteristic $p > 0$. If there is a Zariski dense adelic point on $X$ which survives $p^n$-descent for all $n \ge 1$, then $X$ is isotrivial.
\end{Proposition}

The proof of this proposition will be given at the end of this section.

\begin{Definition}
	Let $L \subset K$ be a subfield. We say that $X$ is \defi{defined over} $L$ if there exists $X_0/L$ such that $X \simeq X_0\times_K L$. We say that $X$ is \defi{definable over} $L$ if there exists $X_0/L$ such that $X \times_K \Kbar \simeq X_0 \times_L \Kbar$, where $\Kbar$ denotes an algebraic closure of $K$ containing $L$.
\end{Definition}

For an abelian variety $A/K$, multiplication by $p^n$ factors as
\[
	A \stackrel{F^n}{\to} A^{(p^n)} \stackrel{V^n}\to A\,,
\]
where $F^n$ and $V^n$ are the $n$-fold compositions of the absolute Frobenius and Verschiebung isogenies.

\begin{Lemma}\label{lem:Frob}
	Suppose $X$ contains a Zariski dense adelic point which lifts to a $p^n$-covering $Y' \to X$ and let $Y \to X$ be the torsor under $\ker(V^n:A^{(p^n)} \to A)$ through which it factors. Then $Y_{\red}$ is geometrically reduced and definable over $K^{p^n}$.
\end{Lemma}

\begin{proof}
	Let $(P_v)_v \in X(\A_K)$ be the given adelic point and let $Y' \to X$ be the $p^n$-covering to which $(P_v)_v$ lifts. By passing to a separable extension of $K$ (which is harmless because $(K^{\textup{sep}})^p \cap K = K^p$ and \cite[Lemma 1.5.11]{MR2708514}) we can assume $Y' \to X$ is the pullback of multiplication by $p^n$ on $A$. In particular, it factors through the $n$-fold Frobenius morphism $F^n : A \to A^{(p^n)}$ and we have a commutative diagram with $Y$ the torsor in the statement,
	\[
		\xymatrix{
			Y'_{\red} \ar[r]\ar[d] & Y_{\red} \ar[r]\ar[d] & X \ar@{=}[d]\\
			Y' \ar[r]\ar[d] & Y \ar[r]\ar[d] & X \ar[d] \\
			A \ar[r]^{F^n} & A^{(p^n)} \ar[r]^{V^n}& A\,.
		}
	\]
	Let $(Q_v)_v \in Y'(\A_K)$ denote a lift of $(P_v)_v$. For any $v$, the point $Q_v : \Spec(K_v) \to Y'$ factors through $Y'_{\red}$, because $\Spec(K_v)$ is reduced. So $(Q_v)_v$ is also a Zariski dense adelic point on $Y'_{\red}$. Its image $(R_v)_v$ in $Y_{\red}(\A_K)$ is a Zariski dense adelic point and by commutativity of the diagram the image of $(R_v)_v$ in $A^{(p^n)}$ lies in $F^n(A(\A_K))$. In particular, for each $v$, the point $R_v$ lies in $A^{(p^n)}(K_v^{p^n})$. It then follows from the proof of \cite[Lemma 1]{AV} that $Y_{\red}$ is defined over $K^{p^n}$ and is geometrically reduced. Below is an alternative argument using \cite{V}, in particular, the last paragraph.

We show that that $Y_{\red}$ is defined over $K^{p^n}$ and is geometrically reduced. Assume $n=1$, which is enough, as the argument can be repeated $n$ times.
Let $U$ be an affine open subset of $Y_{\red}$ and $f$ a function defined on an affine open set of
$A^{(p^n)}$ which vanishes on $U$. We have that $f(R_v)=0$ and differentiating this
equation with respect to a derivation $\delta$ on $K$ with kernel $K^p$, gives
$f^{\delta}(R_v) = 0$. Since $(R_v)_v$ is Zariski dense on $Y_{\red}$, we 
conclude that $f^{\delta}$ also vanishes on $U$. This means that $\delta$
extends to a vector field on a spreading out of $Y_{\red}$ and we conclude
via \cite[Lemma 1]{V}.
\end{proof}

\begin{Remark}
	From the above proof, if $Y_{\red}$ is not defined over $K^{p}$, some $f^{\delta}$ does not vanish on $Y_{\red}$ and the equation $f^{\delta}=0$ defines a proper Zariski closed subset containing $(R_v)_v$.
\end{Remark}

\begin{Lemma}\label{lem:etale}
	If $X' \to X$ is a torsor under an \'etale group scheme and $X'$ is definable over $K^{p^n}$, then $X$ is definable over $K^{p^n}$.
\end{Lemma}

\begin{proof}
	\cite[Lemma 2]{V} proves this for Galois covers. This gives the result, since taking a separable extension to trivialise the
	Galois action on the \'etale group scheme is harmless.
\end{proof}

\begin{Lemma}\label{lem:connected}
	Suppose $Y_i \subset A_i$ are geometrically integral curves contained in abelian varieties $A_i$ over $K$, for $i = 1,2$. Suppose there is an isogeny $A_1 \to A_2$ restricting to a generically purely inseparable map $Y_1 \to Y_2$. If $Y_1$ and $A_1$ are definable over $K^{p^n},$ then $Y_2$ is definable over $K^{p^n}$.
\end{Lemma}

\begin{proof}
	Passing to a finite separable extension we can assume $Y_1$ is defined over $K^{p^n}$. In particular, $Y_1$ is defined over $K^p$, so the argument in \cite[Theorem A(2)]{AV} shows that $Y_2$ is defined over $K^p$. Replacing $K$ with $K^p$ and repeating $n$ times we find that $Y_2$ is defined over $K^{p^n}$.
\end{proof}

\begin{proof}[Proof of Proposition~\ref{prop:main}]
Let $P := (P_v)_v \in X(\A_K)$ be a Zariski dense adelic point that survives $p^n$-descent for all $n \ge 1$.

Let $n \ge 1$ and let $Y' \to X$ be a $p^n$-covering to which $P$ lifts. By Lemma~\ref{lem:Frob}, $Y' \to X$ factors through a torsor $Y \to X$ under the kernel of $V^n : A^{(p^n)} \to A$, with $Y_{\red}$ geometrically reduced and definable over $K^{p^n}$. We can factor $V^n = V_e \circ V_c$, with $V_c$ an isogeny whose kernel is a connected abelian $p$-group scheme and $V_e$ \'etale. Let $Y \to X_e \to X$ be the corresponding factorization of $Y \to X$. Since $X_e \to X$ is \'etale and $X$ is smooth, $X_e$ is geometrically integral. The isogeny $V_c$ restricts to a morphism $Y_{\red} \to X_e$ which is generically purely inseparable, so $X_e$ is definable over $K^{p^n}$ by Lemma~\ref{lem:connected}. Then $X$ is definable over $K^{p^n}$ by Lemma~\ref{lem:etale}.

Since $P$ survives $p^n$-descent for all $n$, we conclude that $X$ is definable over $K^{p^n}$ for all $n \ge 1$. This implies that $X$ is isotrivial  (see the discussion in \cite[Section 0]{MR3618571}).
\end{proof}

\section{Rational points on finite subschemes of abelian varieties}
	
\begin{Proposition}\label{prop:ZsubA} Let $Z \subset A$ be a finite subscheme of an abelian variety defined over a global function field $K$.  Then 
	\[
		Z(K) = Z(\A_K) \cap \overline{A(K)} = Z(\A_K) \cap A(\A_K)^{\Br}\,.
	\]
\end{Proposition}

\begin{proof}
	By~\cite[Theorem E]{PoonenVoloch} we have $Z(K) \subset \overline{A(K)} \subset A(\A_K)^{\Br} \subset \widehat{\Sel}(A).$ So it suffices to show that $Z(\A_K) \cap \widehat{\Sel}(A) \subset Z(K)$. As in the proof of \cite[Prop. 3.9]{PoonenVoloch}, it suffices to show that this holds after a finite separable extension, so we can assume that $Z$ consists of is a finite set of $K$-points. 
	
	Replacing $K$ by a further finite separable extension if needed, we can also assume that $A[n]$ is a constant group scheme for some $n$ prime to $p$ and that the N\'eron model of $A$ has semiabelian connected component. In the appendix by D. R\"{o}ssler it is shown that, under these hypotheses, there exists an \'etale isogeny $f:A \to B$ and an isogeny $g:B\to B$ of degree $> 1$ such that $\ker(g)(K^{\textup{sep}}) = 0$. Let $W = f(Z) \subset B$. If $B(K^{\textup{sep}})[p] = 0$, then \cite[Proposition 5.3]{PoonenVoloch} gives that $W(\A_K) \cap \widehat{\Sel}(B) = W(K) $. Working with the given endomorphism $g$ instead of multiplication by $p$, the argument there can be adapted to give the same conclusion (Details are given in Lemma~\ref{lem:details} below).	
	
	Now suppose $P \in Z(\A_K) \cap \widehat{\Sel}(A)$. It follows from the definition of the Selmer groups that $f\left(\widehat{\Sel}(A)\right) \subset \widehat{\Sel}(B)$. So $f(P) \in W(\A_K) \cap \widehat{\Sel}(B) = W(K)$. For any $v \in \Omega_K$, the $v$-adic component of $P$ is the image of some $Q_v \in Z(K)$. The adelic point $P - Q_v \in A(\A_K)$ lies in the kernel of $f$ and in $\widehat{\Sel}(A)$. So $P-Q_v\in \widehat{\Sel}(A)_{\textup{tors}}$. By \cite[Lemma 5.1]{PoonenVoloch} this implies that $P-Q_v \in A(K)$. So $P \in A(K)$.
\end{proof}

Here are details of the claimed analogue of \cite[Proposition 5.3]{PoonenVoloch} used in the proof above.

\begin{Lemma}\label{lem:details}
	Let $W \subset B$ be a finite subscheme of an abelian variety defined over $K$. Suppose there exists an endomorphism $g : B \to B$ of degree $> 1$ such that $B(K^{\textup{sep}})[g] = 0$. Then $W(\A_K) \cap \widehat{\Sel}(B) = W(K)$.
\end{Lemma}

\begin{proof}
	We have
	\[
		W(K) \subset B(K) \subset B(\A_K)^{\Br} \subset \widehat{\Sel}(B)\,.
	\]
	So it suffices to show that $W(\A_K) \cap \widehat{\Sel}(B) \subset W(K)$. Moreover we can assume $W = W(K)$ as in \cite[Prop. 3.9]{PoonenVoloch}.
	
	Suppose $P \in W(\A_K) \cap \widehat{\Sel}(B)$. For any $v \in \Omega_K$, the $v$-adic component of $P$ is the image of some point $Q_v \in W(K)$, and $P - Q_v \in \widehat{\Sel}(B)$ maps to $0$ in $B(K_v)^{(g)} := \varprojlim_nB(K_v)/	g^n(B(K_v))$. In particular, $P - Q_v$ is in the kernel of $\Sel^{(g)}(B) \to B(K_v)^{(g)}$ where $\Sel^{(g)}(B)$ denotes the inverse limit of the Selmer groups corresponding to the isogenies $g^n$ for $n \ge 1$. Below we show that this map is injective, so the image of $P- Q_v$ in  $\Sel^{(g)}(B)$ is $0$.
	
	Since this holds for any $v$, if $v'$ is any other prime we have 
	\[
	Q_{v'} - Q_v \in \ker\left( B(K) \to \varprojlim_{n}B(K)/g^nB(K) \hookrightarrow \Sel^{(g)}(B)\right)\,.
	\]
	In other words, $(Q_v - Q_{v'}) \in \cap_{n \ge 1}g^nB(K)$. Since $B(K)$ is finitely generated, this implies that $(Q_v - Q_{v'}) \in B(K)_{\textup{tors}}$. Again, since this holds for all $v$ we see that $R := P - Q_v \in \widehat{\Sel}(B)_{\textup{tors}}$. By \cite[Lemma 5.1]{PoonenVoloch} this implies that $P-Q_{v} \in B(K)$. So $P \in W(K)$.

	It remains to prove that $\Sel^{(g)}(B) \to B(K_v)^{(g)}$ is injective. For this it suffices (as in \cite[Proof of 5.2]{PoonenVoloch}) to prove injectivity of $\Sel'^{(g)}(B) \to B(K'_v)^{(g)}$, where $K_v' \subset K^{\textup{sep}}$ denotes the Henselization with respect to $v$ and $\Sel'^{(g)}(B)$ is defined in the same way as $\Sel^{(g)}(B)$ but using $K_v'$ instead of $K_v$. Let $b \in \ker\left( \Sel'^{(g)}(B) \to B(K_v')^{(g)} \right)$ and let $b_M$ denote its image in $\Sel'^{g^M}(B) \subset \HH^1(K_v',B[g^M])$. Then the image of $b_M$ under 
	\[
		\Sel'^{g^M}(B) \to \frac{B(K_v')}{g^MB(K_v')} \subset \HH^1(K_v',B[g^M]) \to \HH^1(K^{\textup{sep}},B[g^M])
	\]
	is $0$. The inflation-restriction sequence
	\[
		0 \to \HH^1(\Gal(K^{\textup{sep}}/K),B(K^{\textup{sep}})[g^M]) \to \HH^1(K,B[g^M]) \to \HH^1(K^{\textup{sep}},B[g^M])
	\]
	shows that $b_M$ comes from an element of $\HH^1(\Gal(K^{\textup{sep}}/K),B(K^{\textup{sep}})[g^M])$. But this group is trivial since $B(K^{\textup{sep}})[g^M]=0$. Since this holds for all $M \ge 1$, we conclude that $b = 0$.
\end{proof}

\section{Proofs of the theorems}

\begin{proof}[Proof of Theorem~\ref{thm:MW}]
By~\cite[Theorem E]{PoonenVoloch} we have $\overline{A(K)} \subset A(\A_K)^{\Br} \subset \widehat{\Sel}(A) \subset A(\A_K)$. intersecting with $X(\A_K)$ we have
	\[
		X(\A_K) \cap \overline{A(K)} \subset X(\A_K) \cap A(\A_K)^{\Br} \subset X(\A_K) \cap \widehat{\Sel}(A)\,,
	\]
	where the rightmost set consists of the adelic points on $X$ which survive $N$-descent for all $N \ge 1$ (relative to the embedding $X \subset A$). In particular, any $P \in X(\A_K) \cap \overline{A(K)}$ survives $p^n$-descent for all $n \ge 1$ . Since $X$ is not isotrivial, Proposition~\ref{prop:main} implies that there is a finite subscheme $Z \subset X \subset A$ which contains $P$. Then $P \in Z(\A_K) \cap \overline{A(K)} = Z(\A_K) \cap A(\A_K)^{\Br} = X(\A_K) \cap \widehat{\Sel}(A) = Z(K)$, where the final equality is Proposition~\ref{prop:ZsubA}.
\end{proof}

\begin{Remark}\label{rem:Sel}
	The preceding proof shows that for $X \subset A$ as in Theorem~\ref{thm:MW} we have 
	\[
		X(K) = X(\A_K) \cap \overline{A(K)} = X(\A_K) \cap A(\A_K)^{\Br} = X(\A_K) \cap \widehat{\Sel}(A)\,.
	\]
\end{Remark}

\begin{proof}[Proof of Theorem~\ref{thm:BM}]
	Let $X/K$ be as in the statement and let $J = \Jac(X)$ be its Jacobian. It suffices to show that $X(\A_K)^{\Br} \subset X(K)$. Passing to some finite separable extension $L/K$ we can embed $X_L$ in $J_L$ via the Abel-Jacobi map corresponding to an $L$-rational point. If $P \in X(\A_K)^{\Br}$, then its image under the inclusion $X(\A_K) \subset X(\A_L) = X_L(\A_L)$ is orthogonal to $\Br(X_L)$ by \cite[Lemma 3.1]{CreutzViray}. By functoriality of the Brauer pairing and Remark~\ref{rem:Sel} we have $X_L(\A_L)^{\Br} \subset X_L(\A_L) \cap J_L(\A_L)^{\Br} = X_L(L).$ Then $P$ is in $X(\A_K) \cap X(L)$ which is equal to $X(K)$ by~\cite[Lemma 3.2]{PoonenVoloch}.
\end{proof}

\begin{Remark}\label{rem:Poonen}
	In the proof of Theorem~\ref{thm:BM} just given \cite[Lemma 3.1]{CreutzViray} and~\cite[Lemma 3.2]{PoonenVoloch} allow us to pass to an extension over which $X$ can be embedded in its Jacobian. Alternatively one can use the following construction suggested to one of us by Poonen. Restriction of scalars gives a map $\Res_{L/K}(X_L) \to \Res_{L/K}(J_L)$. Composing this with the canonical map $X \to \Res_{L/K}(X_L)$ gives a closed immersion $X \to A$ into the abelian variety $A := \Res_{L/K}(J_L)$ over $K$. To prove Theorem~\ref{thm:BM} one can then apply Remark~\ref{rem:Sel} to $X \subset A$.
\end{Remark}

\section*{Acknowledgements}
The authors were supported by the Marsden Fund Council, managed by Royal Society Te Ap\=arangi. We are grateful to Damian R\"ossler for allowing us to include his note~\cite{RosslerNote} (which was originally circulated in 2012 as an improvement to a result in \cite{RosslerI}) as an appendix to this paper. We thank Bjorn Poonen for suggesting the construction in Remark~\ref{rem:Poonen}.

\appendix
\section{On abelian varieties with an infinite group of separable $p^\infty$-torsion points (by Damian R\"{o}ssler)}

If $n\in\N$, we write $[n]$ for the
multiplication by $n$ endomorphism on an abelian variety. If
$h$ is a finite endomorphism of an abelian variety $A$ over a field $L$, we write
$$
A(L)[h^\ell]:=\{x\in A(L)\,|\,h^{\circ \ell}(x)=0\}
$$
and
$$
A(L)[h^\infty]:=\{x\in A(L)\,|\,\exists n\in\N: h^{\circ n}(x)=0\}.
$$
Here $h^{\circ n}(x):=h(h(h(\cdots (x)\cdots)))$, where there are $n$ pairs of brackets.
The notation $A(L)[n^\ell]$ (resp. $A(L)[n^\infty]$) will be a shorthand for
$A(L)[[n]^\ell]$ (resp. $A(L)[[n]^\infty]$).

Let now $K_0$ be the function field of a smooth and proper curve $U$ over a finite field $\F$ of characteristic $p>0$. Let $B$ be an abelian variety over $K_0$.  Suppose that for some $n>3$ prime to $p$, the group scheme $B[n]$ is constant and that
the N\'eron model of $B$ over $U$ has a semiabelian connected component.

\begin{Proposition}
There exists
\begin{itemize}
\item an abelian variety $C$ over $K_0$;
\item an \'etale $K_0$-isogeny $\phi:B\to C$;
\item an \'etale $K_0$-isogeny $f:C\to C$;
\item a $K_0$-isogeny $g:C\to C$;
\item a natural number $r\geq 0$
\end{itemize}
such that
\begin{itemize}
\item[\rm (a)] $g\circ f=[p^r]$ and $g\circ f=f\circ g$;
\item[\rm (b)] $C(K_0^\sep)[p^\infty]=C(K_0^\sep)[f^\infty]=C(\bar K_0)[f^\infty]$;
\item[\rm (c)] $C(K_0^\sep)[g^\infty]=0$.
\end{itemize}
\end{Proposition}
\begin{proof}
For $\ell\geq 0$, define inductively
$$
B_0:=B
$$
and
$$
B_{\ell+1}:=B_\ell/(B_\ell(K_0^\sep)[p]).
$$
For $\ell_2\geq\ell_1$, let $\phi_{\ell_1,\ell_2}:B_{\ell_1}\to B_{\ell_2}$ be the (\'etale!) morphism
obtained by composing the natural morphisms $B_{\ell_1}\to B_{\ell_1+1}\to\cdots\to B_{\ell_2}$.
We first \underline{claim} that
\begin{equation}
(\ker\, \phi_{\ell_1,\ell_2})(K_0^\sep)=B_{\ell_1}(K_0^\sep)[p^{\ell_2-\ell_1}]
\label{oeq}
\end{equation}
We prove the claim by induction on $\ell_2-\ell_1$. For $\ell_2-\ell_1\leq 1$, the claim is true
by definition. Suppose that $\ell_2-\ell_1\geq 1$. Let $x\in B(K_0^\sep)[p^{\ell_2-\ell_1}]$. Then
$[p^{\ell_2-\ell_1-1}](x)\in B(K_0^\sep)[p]$ and thus
$$
\phi_{\ell_1,\ell_1+1}([p^{\ell_2-\ell_1-1}](x))=[p^{\ell_2-\ell_1-1}](\phi_{\ell_1,\ell_1+1}(x))=0.
$$
Applying the inductive assumption to $\phi_{\ell_1,\ell_1+1}(x)$, we see that
$\phi_{\ell_1+1,\ell_2}(\phi_{\ell_1,\ell_1+1}(x))=\phi_{\ell_1,\ell_2}(x)=0$. This proves that
$(\ker\, \phi_{\ell_1,\ell_2})(K_0^\sep)\supseteq B_{\ell_1}(K_0^\sep)[p^{\ell_2-\ell_1}]$. To prove
the opposite inclusion, let $x\in (\ker\, \phi_{\ell_1,\ell_2})(K_0^\sep)$. We compute
$$
\phi_{\ell_1,\ell_2}(x)=\phi_{\ell_1+1,\ell_2}(\phi_{\ell_1,\ell_1+1}(x))=0,
$$
which implies (by the inductive hypothesis) that
$$
[p^{\ell_2-\ell_1-1}](\phi_{\ell_1,\ell_1+1}(x))=\phi_{\ell_1,\ell_1+1}([p^{\ell_2-\ell_1-1}](x))=0,
$$
which in turn implies that $[p]([p^{\ell_2-\ell_1-1}](x))=[p^{\ell_2-\ell_1}](x)=0$. This proves that
$(\ker\, \phi_{\ell_1,\ell_2})(K_0^\sep)\subseteq B_{\ell_1}(K_0^\sep)[p^{\ell_2-\ell_1}]$ and completes the proof
of the claim.

Now we know that by the reasoning made in the last page of \cite{RosslerI},
that there are only finitely many isomorphism classes of
abelian varieties over $K_0$ in the sequence $\{B_\ell\}_{\ell\in\N}$.
Let $C$ be an abelian variety over $K_0$, which appears at least twice  in the sequence $\{B_\ell\}_{\ell\in\N}$.
Let $n_2>n_1$ be such that $C\simeq B_{n_1}\simeq B_{n_2}$.
Then by construction (under the identification $C=B_{n_1}$)
\begin{equation}
\phi_{n_1,n_2}^{\circ\ell}=\phi_{n_1,n_1+\ell\cdot(n_2-n_1)}
\label{ooeq}
\end{equation}
for any $\ell\geq 1$ and thus
\begin{equation}
C(K_0^\sep)[p^\infty]=C(K_0^\sep)[\phi_{n_1,n_2}^\infty]
\label{impeq}
\end{equation}
Now define $f:=\phi_{n_1,n_2}$ and $r:=n_2-n_1$.
Define $g$ as the only $K_0$-isogeny such that $g\circ f=[p^r]$.

Notice then that the identity $g\circ f=[p^r]$ implies the identity $f\circ g=[p^r]$. To see this last fact directly, recall first that
there are natural injection of rings
$$
\End_{K_0}(C)\hookrightarrow\End_{\bar K_0}(C_{\bar K_0})\hookrightarrow \End_{\Z_t}(T_t(C(\bar K_0)))\hookrightarrow
\End_{\Q_t}(T_t(C(\bar K_0))\otimes
\Q_t)
$$
where $T_t(C(\bar K_0))$ is the classical Tate module of $C_{\bar K_0}$ and $t>0$ is some prime
number $\not=p$. Now if $M$ and $N$ are two square matrices of the same size with coefficients
in a field of characteristic $0$, such that $M\cdot N=p^r$, then $p^{-r}N$ is the inverse matrix of $M$ and thus
$N\cdot M=p^r$. This fact combined with the existence of the above injections implies that
$f\circ g=[p^r]$ if  $g\circ f=[p^r]$.

We have already proven (a). Point (b) is contained in equation \eqref{impeq}.

We now prove (c). Suppose that for some $\ell\geq 0$ and some $x\in C(K_0^\sep)$,
we have $g^{\circ l}(x)=0$. Let $y\in (f^{\circ\ell})^{-1}(x)\subseteq C(K_0^\sep)$. Then
$g^{\circ l}(f^{\circ l}(y))=[p^{r\ell}](y)=0$. Hence $f^{\circ l}(y)=0=x$ by \eqref{oeq} and \eqref{ooeq}.
\end{proof}

\section*{References}

%

%\end{bibdiv}

\end{document}